\documentclass[11pt,reqno]{amsart}
\usepackage{amssymb,amsmath,tabularx}
\usepackage{amsthm,verbatim,eucal}
\usepackage[all]{xy}
\usepackage[bookmarks=true]{hyperref}

\newtheorem{thm}{Theorem}
\newtheorem{lem}[thm]{Lemma}
\newtheorem{prop}[thm]{Proposition}
\newtheorem{cor}[thm]{Corollary}
\theoremstyle{definition}

\theoremstyle{remark}
\newtheorem*{note}{Note}

\newcommand{\Lie}{\operatorname{\mathrm{Lie}}}

\newcommand{\sym}[1]{\mathfrak{S}_{#1}}
\newcommand{\Ind}{\operatorname{\mathrm{Ind}}}

\newcommand{\F}{\mathbb{F}}
\newcommand{\Ext}{\operatorname{Ext}}
\newcommand{\E}{\mathcal{E}}
\newcommand{\cO}{\mathcal{O}}
\newcommand{\ul}{\underline}
\begin{document}
\title{The complexity of the Lie module}

\author{Karin Erdmann}
\address[K. Erdmann]{Mathematical Institute, University of Oxford, Oxford, OX1 3LB, United Kingdom.}
\email[K. Erdmann]{erdmann@maths.ox.ac.uk}

\author{Kay Jin Lim}
\author{Kai Meng Tan}
\address[K. J. Lim and K. M. Tan]{Department of Mathematics, National University of Singapore, Block S17, 10 Lower Kent Ridge Road, Singapore 119076.}
\email[K. J. Lim]{matlkj@nus.edu.sg}
\email[K. M. Tan]{tankm@nus.edu.sg}

\date{August 2011}
\thanks{Supported by EPSRC grant EP/G025487/1 and Singapore Ministry of Education Academic Research Fund R-146-000-135-112.}

\begin{abstract}
We show that the complexity of the Lie module $\Lie(n)$ in characteristic $p$ is bounded above by $m$ where $p^m$ is the largest $p$-power dividing $n$ and, if $n$ is not a $p$-power, is equal to the maximum of the complexities of $\Lie(p^i)$ for $1 \leq i \leq m$.
\end{abstract}

\maketitle

\section{Introduction}

The Lie module of the symmetric group $\sym{n}$ appears in many contexts; in particular it is closely related to the free Lie algebra. Here we take it to be
the left ideal of $\F\sym{n}$ generated by the
`Dynkin-Specht-Wever' element
$$\omega_n = (1-d_2)(1-d_3)\cdots(1-d_n)
$$
where $d_i$ is the $i$-cycle $(i, i-1, \ldots,1)$ and we compose the elements of $\sym{n}$ from right to left. We
write $\Lie(n)= \F\sym{n}\omega_n$ for this module, and we assume $\F$ is an algebraically closed field
of characteristic $p$.

One motivation comes from the work of Selick and Wu \cite{SW}. They reduce
the problem of finding  natural homotopy decompositions of the loop suspension of a $p$-torsion suspension to an algebraic question, and
in this context it is important to know a maximal projective submodule
of $\Lie(n)$ when the field has characteristic $p$.
The Lie module also occurs naturally as homology of configuration spaces, and
in other contexts.
Moreover the representation theory of symmetric groups
over prime characteristic is difficult and many basic questions are
open; naturally occurring representations are therefore of interest and may give new understanding.

In this paper we study homological invariants.  More precisely, we
provide upper bounds for the complexity of $\Lie(n)$.
The complexity of a module may be defined to be
the rate of growth of dimensions in its minimal projective resolution.
A module for the group algebra of a finite group has a
cohomological
variety defined via group cohomology, now known as support variety.
Its  dimension
equals the complexity of the module. The computation of  this variety
can be reduced to the case of maximal elementary Abelian $p$-subgroups.
Modules for elementary Abelian $p$-groups also have
a rank variety which in principle is very explicit, and the
support variety is homeomorphic to the rank variety. Details and references
may be found in Chapter 5 of \cite{B}.
Our results are obtained via this route, that is, we study the action
of maximal elementary Abelian $p$-groups on $\Lie(n)$.

A main result of \cite{BS} provides a
decomposition theorem for the  homogeneous parts $L^n(V)$
of the  free Lie algebra on a vector space $V$ over $\F$.
It shows that its module structure, for arbitrary $r$, can be reduced to the
cases when $n=p^m$ for some $m\geq 1$.
We make use of this theorem.  By work in \cite{LT} it may be transferred
to the context of symmetric groups.
Our main results (Theorem \ref{T:ineq} and \ref{T:p-power}) show that the complexity of $\Lie(n)$ is bounded above
by $m$ where $p^m$ is the largest $p$-power dividing $n$ and, if $n$ is not a $p$-power, is equal to the maximum of the complexities of $\Lie(p^i)$ with $1\leq i\leq m$.  We conjecture our upper bound is in fact an equality, and show that this conjecture is equivalent to the assertion that the complexity of $\Lie(p^m)$ as an $\F\E_m$-module is $m$, where $\E_m$ is a regular elementary Abelian subgroup of $\sym{p^m}$ of order $p^m$.

Computer calculations in \cite{D-al} suggest that
the problem of determining the module structure of $\Lie(p^n)$ explicitly is very hard, but
understanding its rank variety, and complexity, may  help.  These computations can be used to obtain the complexities of $\Lie(8)$ and $\Lie(9)$ in characteristic $2$ and $3$ respectively, and they provide some evidence in support of our conjecture.

The paper is organised as follows:  we give a summary of the background theory in the next section and prove some preliminary results. These include
a result (Proposition \ref{P:}) on the complexity of certain modules
for some wreath products in general.  We prove the main results in Section \ref{S:main}, and conclude the paper with some examples in Section \ref{S:eg}.

\section{Preliminaries}

In this section, we provide the necessary background theory that we require and prove some preliminary results.

Throughout, $\F$ denotes an algebraically closed field of characteristic $p$.

\subsection{Complexities and cohomological varieties of modules}

Let $G$ be a finite group.  Denote by $V_G$ the affine variety defined by the maximum ideal spectrum of the cohomology ring $H^\cdot(G,\F) = \Ext^{\cdot}_{\F G} (\F, \F)$.
Given a finitely generated $\F G$-module $M$, its cohomological variety $V_G(M)$ is defined to the subvariety of $V_G$ consisting of maximal ideals of $H^\cdot(G,\F)$ containing the annihilator of $\Ext^*_{\F G}(M,M)$ (thus $V_G(\F) = V_G$).  The complexity of $M$, denoted by $c_G(M)$, is equal to the
(Krull) dimension of $V_G(M)$.

Let $H$ be a subgroup of $G$.  We write $V_H(M)$ and $c_H(M)$ for the cohomological variety and complexity of $M$ as a $\F H$-module.

We collate together some results relating to complexities and varieties of module which we shall require:

\begin{thm} \label{T:cplx}
Let $G$ be a finite group, and let $M$ be a finitely generated $\F G$-module.
\begin{enumerate}
\item $c_G(M) = 0$ if and only if $M$ is projective.
\item $c_G(M) = \max_E \{ c_E(M) \}$ where $E$ runs over representatives of conjugacy classes of maximal elementary Abelian $p$-subgroups of $G$.
\item If $H$ is a subgroup of $G$, then $c_G(\Ind_H^G (M)) = c_H(M)$.
\item If $N$ is another finitely generated $\F G$-module, then $V_G(M \otimes_{\F} N) = V_G(M) \cap V_G(N)$, and $c_G(M \oplus N) = \max\{c_G(M), c_G(N)\}$.
\end{enumerate}
\end{thm}

\subsection{Rank varieties of modules}

Let $E$ be an elementary Abelian $p$-group, that is $E$ is isomorphic to $(C_p)^k$, and assume $E$ has  generators $g_1,g_2, \dotsc, g_k$.  Let $M$ be a finitely generated $\F E$-module.  For each $\alpha = (\alpha_1, \alpha_2,\dotsc, \alpha_k) \in \F^{k}$ with $\alpha \ne 0$, let $u_{\alpha} = 1+\sum_{i=1}^k \alpha_i (g_i-1) \in \F E$.  Then $(u_{\alpha})^p = 1$.  Write $\left< u_{\alpha} \right>$ for the cyclic group of order $p$ generated by $u_{\alpha}$.  Then the group algebra $\F \left< u_{\alpha} \right>$ is subalgebra of $\F E$.

Let $M$ be a finitely generated $\F E$-module.  The rank variety $V^{\#}_E(M)$ of $M$ is defined as
$$
V^{\#}_E(M)=\{ \alpha \in \F^k \mid \alpha \ne 0,\  M \text{ is non-projective as an $\F\left< u_{\alpha} \right>$-module} \} \cup \{0\}.
$$
This is an affine subvariety of $\F^k$, and is independent of the choice and order of the generators (in the sense that two varieties obtained using different choices of generators are isomorphic).  More importantly, we have:

\begin{thm}
Let $E = (C_p)^k$, and let $M$ be a finitely generated $\F E$-module.  Then $V_E(M)$ and $V_E^{\#}(M)$ are isomorphic as affine varieties.  In particular, $c_E(M) \leq k$.
\end{thm}

\begin{lem} \label{L:}
Let $E = E_1 \times E_2$ be an elementary Abelian $p$-group.  Suppose that $M$ is a finitely generated $\F E$-module such that $M$ is projective as an $\F E_1$-module.  Then $c_{E_1 \times E_2}(M) \leq s$ where $E_2 \cong (C_p)^s$.
\end{lem}

\begin{proof}
Let $E_1 \cong (C_p)^r$, and choose generators $g_1,\dotsc, g_{r+s}$ for $E$ such that $g_1,\dotsc, g_r \in E_1$ and $g_{r+1},\dotsc, g_{r+s} \in E_2$.  Embedding $\F^{r}$ into $\F^{r+s}$ in the obvious way, we have
$$
\F^{r} \cap V^{\#}_E(M) = V^{\#}_{E_1}(M).
$$
Thus,
$$
0 = c_{E_1}(M) = \dim(\F^{r} \cap V^{\#}_E(M)) \geq r + c_E(M) - (r+s)
$$
by \cite[Chapter I, Proposition 7.1]{RH}, so that $c_E(M) \leq s$.
\end{proof}

\subsection{Symmetric groups}

Let $n \in \mathbb{Z}^+$.  Denote by $\sym{n}$ the symmetric group on $n$ letters.  We identify $\sym{n}$ with the permutation group on $\{1,2,\dotsc, n\}$, and we compose the elements in $\sym{n}$ from right to left.  For $m \in \mathbb{Z}^+$ with $m \leq n$, we view $\sym{m}$ as the subgroup of $\sym{n}$ fixing $\{ m+1,m+2,\dotsc, n\}$ pointwise.

Let $r, s \in \mathbb{Z}^+$.  For $1 \leq i \leq s$ and $\sigma \in \sym{r}$, write $\sigma[i] \in \sym{rs}$ for the permutation sending $(i-1)r+j$ to $(i-1)r+\sigma(j)$ for each $1 \leq j \leq r$, and fixing everything else pointwise.  Also, let $\Delta_s\, \sigma = \prod_{i=1}^s \sigma[i]$.  If $H$ is a subgroup of $\sym{r}$, let $H[i] = \{ \sigma[i] \mid \sigma \in H \}$.  For $\tau \in \sym{s}$, write $\tau^{[r]} \in \sym{rs}$ for the permutation sending $(i-1)r + j$ to $(\tau(i)-1)r+j$ for each $1\leq i \leq s$ and $1 \leq j \leq r$.  If $K$ is a subgroup of $\sym{s}$, let $K^{[r]} = \{ \tau^{[r]} \mid \tau \in K \}$.

Let $a_p = (1,2,\dotsc, p) \in \sym{p}$.  For $r \in \mathbb{Z}^+$, let
$$
\E_r = \left< \Delta_{p^{r-1}}\,a_p, \Delta_{p^{r-2}}\, a_p^{[p]}, \dotsc, a_p^{[p^{r-1}]} \right> \subseteq \sym{p^r}.
$$
This is an elementary Abelian $p$-subgroup of $\sym{p^r}$ isomorphic to $(C_p)^r$.  These $\E_r$'s are the building blocks of distinguished representatives of the conjugacy classes of maximal elementary Abelian $p$-subgroups of $\sym{n}$:

\begin{thm}[{\cite[Chapter VI, Theorem 1.3]{AM}}] \label{T:eleab}
Let $n \in \mathbb{Z}^+$, and let $k = \lfloor n/p \rfloor$.  Every maximal elementary Abelian $p$-subgroup of $\sym{n}$ is conjugate to one of the following form:
$$
\prod_{j=1}^m \E_{r_j}[s_j/p^{r_j}]
$$
where $(r_1,r_2,\dotsc, r_m)$ is a decreasing sequence of positive integers such that $\sum_{i=1}^m p^{r_i} = pk$, and $s_j = \sum_{i=1}^j p^{r_i}$.
\end{thm}

\begin{note}
The support of each factor $\E_{r_j}[s_j/p^{r_j}]$ is $\{s_{j-1}+1, s_{j-1}+2,\dotsc, s_j\}$, so that these factors have disjoint support.
\end{note}






\subsection{Wreath products} \label{S:wreath}

Let $G$ be a finite group, and let $n \in \mathbb{Z}^+$.  The wreath product $G \wr \sym{n}$ has underlying set $\{ (g_1,\dotsc, g_n) \sigma \mid g_1,\dotsc, g_n \in G,\ \sigma \in \sym{n} \}$, and it  is the group with group composition defined by
$$
((g_1,\dotsc, g_n) \sigma) \cdot ((g'_1,\dotsc, g'_n) \tau) = (g_1g'_{\sigma^{-1}(1)}, \dotsc, g_ng'_{\sigma^{-1}(n)}) (\sigma \tau).
$$
We identify $\sym{n}$ with the subgroup $\{(1,\dotsc, 1) \sigma \mid \sigma \in \sym{n} \}$ of $G \wr \sym{n}$.

Let $M$ be a finitely generated (non-zero) left $\F G$-module.  Then $M^{\otimes n}$ admits a natural left $\F(G \wr \sym{n})$-action via
\begin{multline*}
((g_1,\dotsc, g_n) \sigma) \cdot (m_1\otimes \dotsb \otimes m_n) = (g_1m_{\sigma^{-1}(1)}) \otimes \dotsb \otimes (g_nm_{\sigma^{-1}(n)}) \\
\quad (g_1,\dotsc, g_n \in G,\ \sigma \in \sym{n},\ m_1,\dotsc, m_n \in M).
\end{multline*}
Suppose that
$$
M = \bigoplus_{i \in I} M(i)
$$
is a decomposition of $M$ as $\F G$-modules.  For each $\mathbf{i} = (i_1,\dotsc, i_n) \in I^n$, write $M(\mathbf{i})$ for the subset $M(i_1) \otimes_{\F} \dotsb \otimes_{\F} M(i_n)$ of $M^{\otimes n}$, so that
$$
M^{\otimes n} = \bigoplus_{\mathbf{i} \in I^n} M({\mathbf{i}}).$$

For each $A \subseteq I^n$, write $M(A)$ for $\bigoplus_{\mathbf{i} \in A} M(\mathbf{i})$ (thus, $M(I^n) = M^{\otimes n}$).  The set $I^n$ admits a natural left action of $\sym{n}$ via place permutation, i.e.\ $\sigma \cdot (i_1,\dotsc,i_n) = (i_{\sigma^{-1}(1)},\dotsc, i_{\sigma^{-1}(n)})$.  We note that the $\F(G \wr \sym{n})$-action on $M^{\otimes n}$ satisfies $\sigma (M(\mathbf{i})) = M(\sigma\cdot \mathbf{i})$ for all $\sigma \in \sym{n}$.  Thus, if $\mathcal{O}$ is a $\sym{n}$-orbit of $I^n$, then $M(\mathcal{O})$ is an $\F(G \wr \sym{n})$-submodule of $M^{\otimes n}$, and
$$M^{\otimes n} = \bigoplus_{\mathcal{O}} M(\mathcal{O})$$
where $\mathcal{O}$ runs over all $\sym{n}$-orbits of $I^n$.

Now suppose that $G$ is a subgroup of a finite group $K$.  Let $N = \Ind_G^K M$ and for each $i \in I$, let $N(i) = \Ind_G^K M(i)$.  Then $N = \bigoplus_{i\in I} N(i)$, and using analogous notations introduced above, we see that
$$N^{\otimes n} = \bigoplus_{\mathcal{O}} N(\mathcal{O}),$$
where $\mathcal{O}$ runs over all $\sym{n}$-orbits of $I^n$, is a decomposition of $N^{\otimes n}$ as $\F(K \wr \sym{n})$-modules.  In addition,
for each $\sym{n}$-orbit $\mathcal{O}$ of $I^n$, we have

\begin{lem} \label{L:wreathisom}
$$N(\mathcal{O}) \cong \Ind_{G \wr \sym{n}}^{K \wr \sym{n}} M(\mathcal{O}).$$
\end{lem}

\begin{proof}
Let $T$ be a set of left coset representatives of $G$ in $K$.  Then $T^n$ is a set of left coset representatives of $G \wr \sym{n}$ in $K \wr \sym{n}$.  The reader may check that the map $(t_1 \otimes v_1) \otimes \dotsb \otimes (t_n \otimes v_n) \mapsto (t_1,\dotsc,t_n) \otimes (v_1 \otimes \dotsb \otimes v_n)$ for $t_1, \dotsc, t_n \in T$, $v_1 \in M(i_1), v_2 \in M(i_2), \dotsc, v_n \in M(i_n)$ with $(i_1,\dotsc, i_n) \in \mathcal{O}$ gives the required isomorphism.
\end{proof}

\begin{prop} \label{P:}
Let $G$ be an Abelian $p'$-subgroup of a finite group $K$.  Let $M$ be a non-zero $\F G$-module, and let $N = \Ind_G^K M$.  Let $n \in \mathbb{Z}^+$, and let $S$ be an $\F\sym{n}$-module, so that $S$ becomes an $\F(K \wr \sym{n})$-module via inflation.  Then
$$c_{K \wr \sym{n}} (N^{\otimes n} \otimes_{\F} S) = c_{\sym{n}} (S).$$
\end{prop}

\begin{proof}
Since $\operatorname{char}(\F) = p \nmid |G|$, we see that $M$ is completely reducible, so that
$$M = \bigoplus_{i \in I} M(i)$$ where, since $G$ is Abelian, each $M(i)$ is one-dimensional and $I\neq \varnothing$.  Let $N(i) = \Ind_G^K M(i)$ for each $i \in I$, so that $N = \bigoplus_{i \in I} N(i)$.  We have, by Lemma \ref{L:wreathisom},
$$
N^{\otimes n} = \bigoplus_{\mathcal{O}} N(\mathcal{O}) \cong \bigoplus_{\mathcal{O}} \Ind_{G \wr \sym{n}}^{K \wr \sym{n}} M(\mathcal{O}),$$
where the sum runs over all $\sym{n}$-orbits $\mathcal{O}$ of $I^n$,
so that
$$
N^{\otimes n} \otimes_{\F} S \cong \left(\bigoplus_{\mathcal{O}} \Ind_{G \wr \sym{n}}^{K \wr \sym{n}} M(\mathcal{O})\right) \otimes_{\F} S
\cong \bigoplus_{\mathcal{O}} \left(\Ind_{G \wr \sym{n}}^{K \wr \sym{n}} (M(\mathcal{O}) \otimes_{\F} S)\right).
$$
Thus $c_{K \wr \sym{n}} (N^{\otimes n} \otimes_{\F} S) = \max_{\mathcal{O}} \{c_{G\wr\sym{n}}(M(\mathcal{O}) \otimes_{\F} S)\}$ by Theorem \ref{T:cplx}(3,4). Since $p \nmid |G|$, we may pick representatives of the conjugacy classes of maximal elementary Abelian $p$-subgroups of $G \wr \sym{n}$ to be subgroups of $\sym{n}$.  For each such representative $E$,
$$
V_E(M(\mathcal{O}) \otimes_{\F} S) = V_E(M(\mathcal{O})) \cap V_E(S) \subseteq V_E(S)
$$
by Theorem \ref{T:cplx}(4), so that $c_E(M(\mathcal{O}) \otimes_{\F} S) \leq c_E(S) \leq c_{\sym{n}}(S)$.  Thus
$$
c_{G \wr \sym{n}} (M(\mathcal{O}) \otimes_{\F} S) = \max_E\{c_E(M(\mathcal{O}) \otimes_{\F} S)\} \leq c_{\sym{n}}(S)
$$
for all $\sym{n}$-orbits $\mathcal{O}$ of $I^n$.  On the other hand, if $i \in I$, then $\mathcal{O}_i = \{ (i,\dotsc, ,i) \}$ is a singleton $\sym{n}$-orbit of $I^n$, and $M(\mathcal{O}_i) = (M(i))^{\otimes n}$ is one-dimensional, on which $\sym{n}$ acts trivially.  Thus,
$$
V_E(M(\mathcal{O}_i) \otimes_{\F} S) = V_E(M(\mathcal{O}_i)) \cap V_E(S) = V_E \cap V_E(S) = V_E(S).
$$
This implies that $c_{G \wr \sym{n}}(M(\mathcal{O}_i) \otimes_{\F} S) = c_{\sym{n}}(S)$, and hence
$$
c_{K \wr \sym{n}}(N^{\otimes n} \otimes_{\F} S) = \max_{\mathcal{O}} \{c_{G\wr\sym{n}}(M(\mathcal{O}) \otimes_{\F} S)\} = c_{\sym{n}}(S).
$$
\end{proof}

\subsection{Lie module}

Denote by $\Lie(n)$ the Lie module for the symmetric group $\sym{n}$.  This is the left ideal of $\F\sym{n}$ generated by the `Dynkin-Specht-Wever' element
$$
\omega_n = (1-d_2) (1-d_3) \dotsm (1-d_n)
$$
where $d_i$ is the descending $i$-cycle $(i,i-1,\dotsc, 1)$ of $\sym{n}$.  (Recall that we compose the elements of $\sym{n}$ from right to left.)

The following lemma about $\Lie(n)$ is well-known, but we are unable to find an appropriate reference in the existing literature.

\begin{lem} \label{L:free}
As an $\F\sym{n-1}$-module, $\Lie(n)$ is free of rank $1$.
\end{lem}

\begin{proof}
It is well-known that $(\omega_n)^2 = n\omega_n$, and $\dim_{\F} (\Lie(n)) = (n-1)!$ (see, for example, \cite[Theorem 8.16]{R}, and \cite[Theorem 5.11]{MKS} with $n_1 = n_2 = \dotsb = 1$).  We claim first that $\omega_n = -\omega_{r-1}d_r \omega_n$ whenever $2 \leq r \leq n$ (note that $\omega_1 = 1$ by definition).  To prove this, we have
\begin{gather*}
\omega_r = \omega_{r-1}(1-d_r) = \omega_{r-1} - \omega_{r-1}d_r, \\
\omega_s \omega_n = \omega_s \omega_s (1- d_{s+1}) \dotsm (1-d_n) = s \omega_s (1- d_{s+1}) \dotsm (1-d_n) = s\omega_n
\end{gather*}
for all $1 \leq s \leq n$.  Thus,
$$
(1+\omega_{r-1}d_r) \omega_n = (1+\omega_{r-1}-\omega_r)\omega_n = \omega_n + (r-1)\omega_n -r \omega_n = 0,
$$
proving the claim.

Now, if $\rho \in \sym{n}$ such that $\rho(1) \ne 1$, say $\rho(r) = 1$, then $\rho \omega_n = -\rho \omega_{r-1}d_r \omega_n$, and $-\rho \omega_{r-1}d_r \in \F \sym{n,1}$, where $\sym{n,1} = \{ \sigma \in \sym{n} \mid \sigma(1) = 1\}$, so that $\rho \omega_n \in \F\sym{n,1} \omega_n$ for all $\rho \in \sym{n}$.
Thus, the (obviously linear) map $\psi : \F \sym{n,1} \to \Lie(n)$ defined by $x \mapsto x\omega_n$ is surjective, and hence bijective by dimension count.  Define $\phi : \F\sym{n-1} \to \Lie(n)$ by $y \to y(1,n) \omega_n$.  Then $\phi$ is clearly an $\F\sym{n-1}$-module homomorphism.  In addition, it is injective since $\psi$ is, and is therefore bijective by dimension count.
\end{proof}

\subsection{Tensor powers and Lie powers}

Let $n,r,s \in \mathbb{Z}^+$.  Let $V$ be a finite-dimensional vector space over $\F$.  If $V$ is a left module for the Schur algebra $S(n,r)$, then the tensor power $V^{\otimes s}$ is naturally a left $S(n,rs)$-module.  In addition, $V^{\otimes s}$ admits a commuting right action of $\sym{s}$ by place permutation.  The Lie power $L^s(V)$ of $V$ may be defined as $(V^{\otimes s})\omega_s$ where $\omega_s$ is the Dynkin-Specht-Wever element mentioned in the last subsection; this is a left $S(n,rs)$-submodule of $V^{\otimes s}$.

If $\dim(V) = n$, then $V$ is naturally a left $S(n,1)$-module.  Thus $V^{\otimes s}$ is a $(S(n,s),\F\sym{s})$-bimodule, while $L^s(V)$ is a left $S(n,s)$-submodule of $V^{\otimes s}$.  When $n \geq s$, the Schur functor $f_s$ sends $V^{\otimes s}$ to the $(\F\sym{s},\F\sym{s})$-bimodule $\F\sym{s}$, and it sends $L^s(V)$ to the left $\F\sym{s}$-module $\Lie(s)$.
The effect of the Schur functor $f_{rs}$ on $V^{\otimes s}$ and $L^s(V)$, when $V$ is a left $S(n,r)$-module and $n \geq rs$, is described in detail in \cite{LT}.  In this paper, we need the latter result:

\begin{thm}[{\cite[Corollary 3]{LT}}] \label{T:Schur}
Let $n, r, s \in \mathbb{Z}^+$ with $n \geq rs$.  Let $V$ be an $S(n,r)$-module.  Then
$$
f_{rs} L^s(V) \cong \Ind_{\sym{r} \wr \sym{s}}^{\sym{rs}} ((f_r(V))^{\otimes s} \otimes_{\F} \Lie(s))
$$
where $\sym{r} \wr \sym{s}$ is identified with the subgroup $(\prod_{i=1}^s \sym{r}[i]) \sym{s}^{[r]}$ of $\sym{rs}$, and it acts on $(f_r(V))^{\otimes s}$ and $\Lie(s)$ via
\begin{align*}
((\sigma_1,\dotsc, \sigma_s)\tau) \cdot (x_1 \otimes \dotsb \otimes x_s) & = (\sigma_1 x_{\tau^{-1}(1)}) \otimes \dotsb \otimes (\sigma_s x_{\tau^{-1}(s)}), \\
((\sigma_1,\dotsc, \sigma_s)\tau) \cdot y &= \tau y,
\end{align*}
for all $\sigma_1,\dotsc, \sigma_s \in \sym{r}$, $\tau \in \sym{s}$, $x_1,\dotsc, x_s \in f_r(V)$ and $y \in \Lie(s)$.
\end{thm}

Bryant and Schocker proved a remarkable decomposition theorem for the Lie powers:

\begin{thm}[{\cite[Theorem 4.4]{BS}}] \label{T:decomp}
Let $k \in \mathbb{Z}^+$ with $p \nmid k$,  and let $V$ be an $n$-dimensional vector space over $\F$.  For each $m \in \mathbb{Z}_{\geq 0}$, there exists $B_{p^m k}(V) \subseteq L^{p^m k}(V)$ such that $B_{p^mk}(V)$ is a
direct summand of $V^{\otimes p^mk}$ as $S(n,p^mk)$-modules, and
$$
L^{p^r k}(V) = L^{p^r}(B_k(V)) \oplus L^{p^{r-1}}(B_{pk}(V)) \oplus \dotsb \oplus L^1(B_{p^rk}(V))
$$
for all $r \in \mathbb{Z}_{\geq 0}$.
\end{thm}

We note that if $k>1$ and $n\geq p^mk$ then $B_{p^ik}(V)$ is non-zero for $0\leq i\leq m$; this is implicit in \cite{BS}.

The $S(n,p^mk)$-submodules $B_{p^mk}(V)$ of $L^{p^mk}(V)$ are further studied in \cite{BJ} and \cite{BDJ}.  In particular, they give the following description for $B_{p^mk}(V)$.  As mentioned in the beginning of this subsection, $\sym{k}$ acts on $V^{\otimes k}$ from the right by place permutation.  Let $a_k = (1,2,\dotsc, k) \in \sym{k}$.  For each $k$-th root of unity $\delta$ in $\F$ (which is algebraically closed, with characteristic $p$ coprime to $k$), let $(V^{\otimes k})_{\delta}$ denote the $a_k$-eigenspace of $V^{\otimes k}$ with eigenvalue $\delta$.

\begin{thm}[{\cite[Theorem 2.6]{BDJ}}] \label{E:B}
Let $m,k \in \mathbb{Z}_{\geq 0}$ with $k>1$ and $p \nmid k$, and let $V$ be an $n$-dimensional vector space where $n \geq p^mk$.  Then
\begin{equation*}
B_{p^m k}(V) \cong \bigoplus_{(\delta_1,\dotsc,\delta_{p^m}) \in \Omega} (V^{\otimes k})_{\delta_1} \otimes \dotsb \otimes (V^{\otimes k})_{\delta_{p^m}}
\end{equation*}
for some (fixed) non-empty subset $\Omega$ of the set of $p^m$-tuples of $k$-th roots of unity.
\end{thm}

We note that $(V^{\otimes k})_{\delta} \cong V^{\otimes k} \otimes_{\F \left< a_k \right>} \F_{\delta}$ as left $S(n,k)$-modules, where $\F_{\delta}$ denotes the one-dimensional left $\F\left<a_k\right>$-module in which $a_k$ acts via multiplication by the scalar $\delta$.

\begin{cor} \label{C:B}
Keep the notations in Theorem \ref{E:B}.  Then
$$
f_{p^mk}(B_{p^mk}) \cong \bigoplus_{(\delta_1,\dotsc,\delta_{p^m}) \in \Omega} \Ind_{\left<a_k \right>^{p^m}}^{\sym{p^mk}} \left(\bigotimes_{j=1}^{p^m} \F_{\delta_j} \right),
$$
where $\left<a_k\right>^{p^m}$ is identified with the subgroup $\prod_{j=1}^{p^m} \left<a_k\right>[j]$ of $\sym{p^mk}$.

In particular, $f_{p^mk}(B_{p^mk})$ is a non-zero $\F\sym{p^mk}$-module induced from $\left<a_k\right>^{p^m}$.
\end{cor}

\begin{proof}
By Theorem \ref{E:B} and \cite[2.5, Lemma]{DE}, we have
\begin{align*}
f_{p^mk}(B_{p^mk}(V)) &\cong \bigoplus_{(\delta_1,\dotsc,\delta_{p^m}) \in \Omega} \Ind_{(\sym{k})^{p^m}}^{\sym{p^mk}} \left(\bigotimes_{j=1}^{p^m} f_k(V^{\otimes k} \otimes_{\F\left<a_k \right>} \F_{\delta_j})\right) \\
&\cong \bigoplus_{(\delta_1,\dotsc,\delta_{p^m}) \in \Omega} \Ind_{(\sym{k})^{p^m}}^{\sym{p^mk}} \left(\bigotimes_{j=1}^{p^m} \F\sym{k} \otimes_{\F\left<a_k \right>} \F_{\delta_j}\right) \\
&= \bigoplus_{(\delta_1,\dotsc,\delta_{p^m}) \in \Omega} \Ind_{(\sym{k})^{p^m}}^{\sym{p^mk}} \left(\bigotimes_{j=1}^{p^m} \Ind_{\left< a_k\right>}^{\sym{k}} \F_{\delta_j}\right) \\
&= \bigoplus_{(\delta_1,\dotsc,\delta_{p^m}) \in \Omega} \Ind_{\left< a_k \right>^{p^m}}^{\sym{p^mk}} \left(\bigotimes_{j=1}^{p^m} \F_{\delta_j}\right).
\end{align*}
\end{proof}

\section{Main results} \label{S:main}

In this section, we prove the main results of this paper.

\begin{lem} \label{L:simple}
Let $n \in \mathbb{Z}^+$, and let $(r_1,r_2,\dotsc, r_t)$ be a weakly decreasing sequence of positive integers such that $\sum_{i=1}^t p^{r_i} = p\lfloor n/p \rfloor$.  For each $j = 1,\dotsc, t$, let $s_j = \sum_{i=1}^j p^{r_i}$.  Let $E = \prod_{j=1}^t \E_{r_j}[s_j/p^{r_j}]$.  Then
$$
c_E(\Lie(n))
\begin{cases}
= 0, &\text{if } p \nmid n; \\
\leq r_t, &\text{if } p \mid n.
\end{cases}
$$
\end{lem}

\begin{proof}
If $p \nmid n$, then $E \subseteq \sym{n-1}$. Since $\Lie(n)$ is free of rank 1 as an $\F\sym{n-1}$-module by Lemma \ref{L:free}, we see that $\Lie(n)$ is projective as an $\F E$-module.  Thus $c_E(\Lie(n)) = 0$ by Theorem \ref{T:cplx}(1).

If $p \mid n$, let $E' = \prod_{j=1}^{t-1} \E_{r_j} [s_j/p^{r_j}]$, so that $E = E' \times \E_{r_t}[n/p^{r_t}]$.  Since $E' \subseteq \sym{n-1}$, we see, as before, that $\Lie(n)$ is projective as an $\F E'$-module, so that $c_{E'}(\Lie(n)) = 0$.  Thus $c_E(\Lie(n)) \leq r_t$ by Lemma \ref{L:}.
\end{proof}

\begin{thm} \label{T:ineq}
We have $c_{\sym{n}}(\Lie(n)) \leq m$, where $p^m \mid n$ and $p^{m+1} \nmid n$.
\end{thm}

\begin{proof}
By Theorems \ref{T:cplx}(2) and \ref{T:eleab}, it suffices to show that $c_E(\Lie(n)) \leq m$ for all $E$ of the form
$$E = \prod_{j=0}^t \E_{r_j} [s_j/p^{r_j}]$$
where $(r_1,r_2,\dotsc, r_t)$ is a weakly decreasing sequence of positive integers such that $\sum_{i=1}^t p^{r_i} = p\lfloor n/p \rfloor$ and $s_j = \sum_{i=1}^j p^{r_i}$.

If $p \nmid n$ (i.e.\ $m = 0$), then $c_E(\Lie(n)) = 0 = m$ by Lemma \ref{L:simple}.

If $p \mid n$, then $c_E(\Lie(n)) \leq r_t$ by Lemma \ref{L:simple}.
Since $\sum_{i=1}^t p^{r_i} = n$ and $(r_1,\dotsc, r_t)$ is weakly decreasing, we see that $p^{r_t} \mid n$.  Thus $r_t \leq m$.
\end{proof}

\begin{thm} \label{T:p-power}
Let $m, k \in \mathbb{Z}^+$ with $p \nmid k$ and $k>1$.  Then
$$
c_{\sym{p^mk }}(\Lie(p^mk)) = \max\{c_{\sym{p^i}}(\Lie(p^i)) \mid 1 \leq i \leq m \}.
$$
\end{thm}

\begin{proof}
By applying the exact and direct-sum-preserving Schur functor $f_{p^mk}$ to Theorem \ref{T:decomp}, we obtain, by Theorem \ref{T:Schur},
$$
\Lie(p^m k) = \bigoplus_{i=0}^m \Ind_{\sym{p^{m-i}k} \wr \sym{p^i}}^{\sym{p^mk}} ((f_{p^{m-i}k}(B_{p^{m-i}k}(V)))^{\otimes p^i} \otimes_{\F} \Lie(p^i)).
$$
By Corollary \ref{C:B} and Proposition \ref{P:},
$$c_{\sym{p^{m-i}k} \wr \sym{p^i}}((f_{p^{m-i}k}(B_{p^{m-i}k}(V)))^{\otimes p^i} \otimes_{\F} \Lie(p^i)) = c_{\sym{p^i}}(\Lie(p^i)).$$
Applying Theorem \ref{T:cplx}(3,4) now completes the proof.
\end{proof}

}

We conjecture that the inequality in Theorem \ref{T:ineq} is in fact an equality.  This assertion is equivalent to the following statements:

\begin{cor}
The following statements are equivalent:
\begin{enumerate}
\item For all $n \in \mathbb{Z}^+$, $c_{\sym{n}}(\Lie(n)) = m$ where $p^m \mid n$ and $p^{m+1} \nmid n$.
\item For all $m \in \mathbb{Z}^+$, $c_{\sym{p^m}}(\Lie(p^m)) = m$.
\item For all $m \in \mathbb{Z}^+$, $c_{\E_m}(\Lie(p^m)) = m$.
\item For all $m \in \mathbb{Z}^+$, $V^{\#}_{\E_m}(\Lie(p^m)) = \F^m$.
\end{enumerate}
\end{cor}

\begin{proof}
(1) and (2) are equivalent by Therorem \ref{T:p-power}, while the equivalence of (2) and (3) follows from Theorem \ref{T:cplx}(2) and Lemma \ref{L:simple}.  That (3) and (4) are equivalent is trivial.
\end{proof}

\section{Some examples} \label{S:eg}

We end the paper with the computation of the complexity of some $\Lie(n)$.  This provides some evidence in support of our conjecture that the inequality in Theorem \ref{T:ineq} is in fact an equality.

\subsection{The case $n=pk$ where $p$ does not divide $k$} By \cite{ES},
any non-projective summand of $\Lie(n)$ has vertex of order $p$ and is therefore
periodic as a module for $\F\sym{n}$. Furthermore such a summand always exists, and hence
$c_{\sym{n}}(\Lie(n))=1$ in this case.


\subsection{The case when $n=2^m$ with $m=2, 3$ and $p=2$}

When $n=8$, the results in \cite{D-al}, which were obtained with the help of computer calculations, can be used to find the complexity.
Recall that any finite-dimensional module $M$ is a direct sum
$M= M^{pf}\oplus M^{pr}$ where $M^{pr}$ is projective, and $M^{pf}$ does
not have a non-zero projective summand. Clearly, $c(M)= c(M^{pf})$.
The projective-free part  $\Lie(8)^{pf}$ of $\Lie(8)$
is indecomposable, with the regular
elementary Abelian subgroup $\E_3$ of order $8$ as its vertex, and a source of dimension $21$.

Generally, if the projective-free part $M^{pf}$ is indecomposable
and has an elementary Abelian vertex $E$ of order $p^m$
 then $V_E^{\#}(M) = V_E^{\#}(S)$ where
$S$ is a source of $M$ in $E$. If $p$ does not divide $\dim_{\F}(S)$, then clearly
$V_E^{\#}(S) = \F^m$. 
Hence $\Lie(8)$ has complexity $3$.

When $n=4$,
it is easy to see that $\Lie(4)$ is isomorphic to $\Omega^{-1}(D)$ where
$D$ is the two-dimensional simple module of $\F\sym{4}$. This has vertex the
regular elementary Abelian subgroup $\E_2$ of order 4, and its source has dimension
$3$. Hence the same argument as for $n=8$ implies that $c_{\sym{4}}(\Lie(4))=2$.

\subsection{The case when $n=9$ and $p=3$}

In this case, there are similar results by \cite{D-al}. Again, $\Lie(9)^{pf}$
is indecomposable, with a vertex the regular elementary Abelian subgroup $\E_2$
of order $9$, and a source of dimension $16$.  The above argument can still be applied to get $c_{\sym{9}}(\Lie(9))=2$.


\end{document}